\documentclass[12pt]{article}
\usepackage{graphicx}
\usepackage{amsmath,amsthm,amssymb,amsfonts,euscript,enumerate}
\usepackage{amsthm}
\usepackage{color,wrapfig}
\usepackage{pxfonts}
\usepackage{verbatim}
\newtheorem{thm}{Theorem}[section]

\newtheorem{lem}[thm]{Lemma}

\newcommand{\1}{\partial}

\newcommand{\3}{\varepsilon}

\newcommand{\R}{{\mathbb R}}

\newcommand{\cD}{{\mathcal D}}

\newcommand{\cX}{{\mathcal X}}

\newcommand{\ve}{\varepsilon}

\begin{document}
\title{Existence of hypercylinder expanders of the inverse mean curvature flow}
\author{Kin Ming Hui\\
Institute of Mathematics, Academia Sinica\\
Taipei, Taiwan, R. O. C.}
\date{Dec 14, 2019}
\smallbreak \maketitle
\begin{abstract}
We will give a new proof of the existence of hypercylinder expander of the inverse mean curvature flow which is a radially symmetric homothetic soliton of the inverse mean curvature flow  in $\mathbb{R}^n\times \mathbb{R}$, $n\ge 2$, of the form $(r,y(r))$ or $(r(y),y)$ where $r=|x|$, $x\in\mathbb{R}^n$, is the radially symmetric coordinate and $y\in \mathbb{R}$. More precisely for any $\lambda>\frac{1}{n-1}$ and $\mu>0$, we will  give a new proof of the existence of a unique even solution $r(y)$ of the equation $\frac{r''(y)}{1+r'(y)^2}=\frac{n-1}{r(y)}-\frac{1+r'(y)^2}{\lambda(r(y)-yr'(y))}$ in $\mathbb{R}$ which satisfies $r(0)=\mu$, $r'(0)=0$ and $r(y)>yr'(y)>0$  for any $y\in\mathbb{R}$. We will  prove that $\lim_{y\to\infty}r(y)=\infty$ and  $a_1:=\lim_{y\to\infty}r'(y)$ exists with $0\le a_1<\infty$. We will also give a new proof of the existence of a constant $y_1>0$ such that $r''(y_1)=0$, $r''(y)>0$ for any $0<y<y_1$ and $r''(y)<0$ for any $y>y_1$. 
\end{abstract}

\vskip 0.2truein



\vskip 0.2truein
\setcounter{equation}{0}
\setcounter{section}{0}

\section{Introduction}
\setcounter{equation}{0}
\setcounter{thm}{0}

Consider a family of immersions $F:M^n\times [0,T)\to\R^{n+1}$ of $n$-dimensional hypersurfaces in $\R^{n+1}$. We say that $M_t=F_t(M^n)$, $F_t(x)=F(x,t)$,  moves by the inverse mean curvature flow if 
\begin{equation*}
\frac{\1}{\1 t}F(x,t)=-\frac{\nu}{H}\quad\forall x\in M^n, 0<t<T
\end{equation*}
where $H(x,t)>0$ and $\nu$ are the mean curvature and unit interior normal of the surface
$F_t$ at the point $F(x,t)$. Recently there are a lot of study on the inverse mean curvature flow by P.~Daskalopoulos, C.~Gerhardt, K.M.~Hui \cite{H}, G.~Huisken, T.~Ilmanen, K.~Smoczyk, J.~Urbas and others \cite{DH}, \cite{G},  \cite{HI1}, \cite{HI2}, \cite{HI3}, \cite{S}, \cite{U}. Although there are a lot of study on the inverse mean curvature flow
on the compact case, there are not many results for the non-compact case. 

Recall that \cite{DLW} a $n$-dimensional submanifold $\Sigma$ of $\R^{n+1}$ with immersion $X:\Sigma\to\R^{n+1}$ and non-vanishing mean curvature $H$ is called a homothetic soliton for the inverse mean curvature flow if there exists a constant $\lambda\ne 0$ such that
\begin{equation}\label{homothetic-soliton-defn}
-\frac{\nu(p)}{H(p)}=\lambda X(p)^{\perp}\quad\forall p\in \Sigma
\end{equation} 
where $X(p)^{\perp}$ is the component of $X(p)$ that is normal to the tangent space $T_{X(p)}(X(\Sigma))$ at $X(p)$.
As proved by G.~Drugan, H.~Lee and G.~Wheeler in \cite{DLW} \eqref{homothetic-soliton-defn} is equivalent to 
\begin{equation}\label{homothetic-soliton-equiv-defn}
-<H\nu, X>=\frac{1}{\lambda}\quad\Leftrightarrow\quad-<\Delta_gX,X>=\frac{1}{\lambda}\quad\forall X\in\Sigma
\end{equation}
where $g$ is the induced metric of the immersion $X:\Sigma\to\R^{n+1}$.
If the homothetic soliton of the inverse mean curvature flow is a radially symmetric solution in $\R^n\times \R$, $n\ge 2$, of the form $(r,y(r))$ or $(r(y),y)$ where $r=|x|$, $x\in\R^n$, is the radially symmetric coordinate, $y\in\R$, then by \eqref{homothetic-soliton-equiv-defn} and a direct computation $r(y)$ satisfies the equation
\begin{equation}\label{r-eqn}
\frac{r''(y)}{1+r'(y)^2}=\frac{n-1}{r(y)}-\frac{1+r'(y)^2}{\lambda(r(y)-yr'(y))}\,\,\,,\quad r(y)>0
\end{equation}
or equivalently $y(r)$ satisfies the equation,
\begin{equation*}
y_{rr}+\frac{n-1}{r}\cdot(1+y_r^2)y_r-\frac{(1+y_r^2)^2}{\lambda(ry_r-y)}=0
\end{equation*}
where $r'(y)=\frac{dr}{dy}$, $r''(y)=\frac{d^2r}{dy^2}$ and $y_r(r)=\frac{dy}{dr}$, $y_{rr}(r)=\frac{d^2y}{dr^2}$ etc. In the paper \cite{DLW} G.~Drugan, H.~Lee and G.~Wheeler stated the existence and asymptotic behaviour of hypercylinder expanders which are homothetic soliton for the inverse mean curvature flow with $\lambda>1/n$. However there are no proof of the existence result in that paper except for the case $\lambda=\frac{1}{n-1}$ and the proof of the  asymptotic behaviour of hypercylinder expanders there are very sketchy. In this paper I will give a new proof of the existence of hypercylinder expanders for the inverse mean curvature flow with $\lambda>\frac{1}{n-1}$. We will also give a new proof of the asymptotic behaviour of these hypercylinder expanders.  

More precisely I will prove the following main results.

\begin{thm}\label{existence_soln-thm}
For any $n\ge 2$, $\lambda>\frac{1}{n-1}$ and $\mu>0$, there exists a unique even solution $r(y)\in C^2(\R)$ of the equation 
\begin{equation}\label{ode-ivp}
\left\{\begin{aligned}
&\frac{r''(y)}{1+r'(y)^2}=\frac{n-1}{r(y)}-\frac{1+r'(y)^2}{\lambda(r(y)-yr'(y))}\,\,\,,\quad r(y)>0,\quad\forall y\in\R\\
&r(0)=\mu,\quad r'(0)=0\end{aligned}\right.
\end{equation}
which satisfies
\begin{equation}\label{f-structure-ineqn}
r(y)>yr'(y)\quad\forall y\in\R
\end{equation}
and 
\begin{equation}\label{r-2nd-derivative-x=0}
r''(0)=\left(n-1-\frac{1}{\lambda}\right)\frac{1}{\mu}.
\end{equation}
\end{thm}

\begin{thm}\label{asymptotic-behaviour-thm}(cf. Theorem 20 of \cite{DLW})
Let $n\ge 2$, $\lambda>\frac{1}{n-1}$, $\mu>0$, and $r(y)\in C^2(\R)$ be the unique solution of \eqref{ode-ivp}. Then 
\begin{equation}\label{r-monotone-increase}
r'(y)>0\quad\forall y>0,
\end{equation}
\begin{equation}\label{r'-infty-+ve}
a_1:=\lim_{y\to\infty}r'(y)\quad\mbox{ exists and }0\le a_1<\infty,
\end{equation}
and 
\begin{equation}\label{r-infty-behaviour}
\lim_{y\to\pm\infty}r(y)=\infty.
\end{equation}
Moreover there exists a constant $y_1>0$ such that
\begin{equation}\label{r''-sign}
\left\{\begin{aligned}
&r''(y)>0\quad\forall 0<y<y_1\\
&r''(y)<0\quad\forall y>y_1\\ 
&r''(y_1)=0.\end{aligned}\right.
\end{equation}
\end{thm}
Since \eqref{ode-ivp} is invariant under reflection $y\to -y$, by uniqueness of solution of ODE the solution of \eqref{ode-ivp} is an even function and Theorem \ref{existence_soln-thm} is equivalent to the following theorem.

\begin{thm}\label{existence_soln-thm2}
For any $n\ge 2$, $\lambda>\frac{1}{n-1}$ and $\mu>0$, there exists a unique solution $r(y)\in C^2([0,\infty))$ of the equation 
\begin{equation}\label{ode-ivp2}
\left\{\begin{aligned}
&\frac{r''}{1+r'{}^2}=\frac{n-1}{r}-\frac{1+r'{}^2}{\lambda(r-yr')}\,\,\,,\quad r(y)>0,\quad\forall y>0\\
&r(0)=\mu,\quad r'(0)=0\end{aligned}\right.
\end{equation}
which satisfies
\begin{equation}\label{f-structure-ineqn2}
r(y)>yr'(y)\quad\forall y>0
\end{equation}
and \eqref{r-2nd-derivative-x=0}.
\end{thm}

\section{Existence and asymptotic behaviour of solution}
\setcounter{equation}{0}
\setcounter{thm}{0} 

In this section we willl prove Theorem \ref{asymptotic-behaviour-thm} and Theorem \ref{existence_soln-thm2}. We first start with a lemma.

\begin{lem}\label{short-time-existence-ode}
For any $n\ge 2$, $\lambda\ne 0$ and $\mu>0$, there exists a constant $y_0>0$ such that  the equation 
\begin{equation}\label{ode-ivp3}
\left\{\begin{aligned}
&\frac{r''}{1+r'{}^2}=\frac{n-1}{r}-\frac{1+r'{}^2}{\lambda(r-yr')}\,\,\,,\quad r(y)>0,\quad\mbox{ in }[0,y_0)\\
&r(0)=\mu,\quad r'(0)=0\end{aligned}\right.
\end{equation}
has a unique solution $r(y)\in C^2([0,y_0))$ which satisfies
\begin{equation}\label{f-structure-ineqn3}
r(y)>yr'(y)\quad\mbox{ in }[0,y_0)
\end{equation}
Moreover \eqref{r-2nd-derivative-x=0} holds.
\end{lem}
\begin{proof}
Uniqueness of solution of \eqref{ode-ivp3} follows from standard ODE theory. Hence we only need to prove existence of solution of \eqref{ode-ivp3}. We will use a modification of the fixed point argument of the proof of Lemma 2.1 of \cite{H} to prove the existence of solution of \eqref{ode-ivp3}. Let $0<\ve<1$. We now define the Banach space 
$$
\cX_\ve:=\left\{(g,h): g, h\in C\left( [0,\ve]; \R\right) \right\}
$$ 
with a norm given by
$$||(g,h)||_{\cX_\ve}=\max\left\{\|g\|_{L^\infty(0, \ve)} ,\|h(s)\|_{L^{\infty}(0,\ve)}\right\}.$$  
For any $(g,h)\in \cX_\ve,$ we define  
$$\Phi(g,h):=\left(\Phi_1(g,h),\Phi_2(g,h)\right),$$ 
where for any $0<y\leq\ve,$
\begin{equation}\label{eq-existence-contraction-map}
\left\{\begin{aligned}
&\Phi_1(g,h)(y):=\mu+\int_0^y h(s)\,ds\\
&\Phi_2(g,h)(y):=\int_0^y (1+h(s)^2)\left(\frac{n-1}{g(s)}-\frac{1+h(s)^2}{\lambda (g(s)-sh(s))}\right)\,ds.
\end{aligned}\right.
\end{equation}
For any $0<\eta\le \mu/4$, let $$\cD_{\ve,\eta}:=\left\{ (g,h)\in \cX_\ve:  ||(g,h)-(\mu,0)||_{\cX_{\ve}}\leq \eta\right\}.$$
Note that $\cD_{\ve,\eta}$  is a closed subspace of $\cX_\ve$. We will show that if $\ve\in(0,1)$ is sufficiently small, the map $(g,h)\mapsto\Phi(g,h)$ will have  a unique  fixed point in $\cD_{\ve,\eta}$.

We first  prove that $\Phi(\cD_{\ve,\eta})\subset \cD_{\ve,\eta}$ if $\ve\in(0,1)$ is sufficiently  small.
Let $(g,h)\in \cD_{\ve,\eta}$. Then 
\begin{align}
&|h(s)|\le \eta\le \mu/4\quad\mbox{ and } \quad \frac{3\mu}{4}\le g(s)\le\frac{5\mu}{4}\quad\forall 0\le s\le\3\label{h-g-bd}\\
\Rightarrow\quad&g(s)-sh(s)\ge\frac{\mu}{2}>0\quad\forall 0\le s\le\3\label{h-g-lower-bd}
\end{align}
and
\begin{equation}\label{phi1-bd}
|\Phi_1(g,h)(y)-\mu|\le\int_0^y |h(s)|\,ds\le\eta\ve\le\eta\quad\forall 0\le y\le\3.
\end{equation}
Hence by \eqref{eq-existence-contraction-map}, \eqref{h-g-bd} and \eqref{h-g-lower-bd},
\begin{equation}\label{phi2-bd}
|\Phi_2(g,h)(y)|\le (1+\eta^2)\left(\frac{4(n-1)}{3\mu}+\frac{2(1+\eta^2)}{|\lambda|\mu}\right)\3\le\eta\quad\forall 0\le y\le\3
\end{equation}
if $0<\3\le\3_1$ where
\begin{equation*}
\3_1=\min\left(\frac{1}{2},\eta(1+\eta^2)^{-1}\left(\frac{4(n-1)}{3\mu}+\frac{2(1+\eta^2)}{|\lambda|\mu}\right)^{-1}\right).
\end{equation*}
Thus  by \eqref{phi1-bd} and \eqref{phi2-bd}, $\Phi(\cD_{\ve,\eta})\subset \cD_{\ve,\eta}$ for any  $0<\ve\le\ve_1$.

We now let $0<\ve\le\ve_1$. Let $(g_1,h_1),(g_2,h_2)\in \cD_{\ve,\eta}$ and $\delta:=||(g_1,h_1)-(g_2,h_2)||_{\cX_\ve}$. Then by \eqref{h-g-bd} and \eqref{h-g-lower-bd},
\begin{equation}\label{hi-gi-bd}
|h_i(s)|\le \eta \le \mu/4, \, \frac{3\mu}{4}\le g_i(s)\le\frac{5\mu}{4}\,\mbox{ and }\, g_i(s)-sh_i(s)\ge\frac{\mu}{2}>0\,\,\forall 0\le s\le\3, i=1,2.
\end{equation}
Hence by \eqref{hi-gi-bd}, we have
\begin{equation}\label{phi1-contraction}
|\Phi_1(g_1,h_1)(y)-\Phi_1(g_2,h_2)(y)|\le\int_0^y |h_1(s)-h_2(s)|\,ds\le\delta\3 \le\frac{\delta}{2}\quad\forall 0\le y\le\3
\end{equation}
and
\begin{align}\label{phi2-difference}
&|\Phi_2(g_1,h_1)(y)-\Phi_2(g_2,h_2)(y)|\notag\\
\le&\int_0^y\left|h_1(s)-h_2(s)\right|\left|h_1(s)+h_2(s)\right|\left(\left|\frac{n-1}{g_1(s)}\right|+\left|\frac{1+h_1(s)^2}{\lambda (g_1(s)-sh_1(s))}\right|\right)\,ds\notag\\
&\qquad +(n-1)\int_0^y(1+h_2(s)^2)\left|\frac{g_2(s)-g_1(s)}{g_1(s)g_2(s)}\right|\,ds\notag\\
&\qquad +\int_0^y(1+h_2(s)^2)\left|\frac{1+h_1(s)^2}{\lambda (g_1(s)-sh_1(s))}-\frac{1+h_2(s)^2}{\lambda (g_2(s)-sh_2(s))}\right|\,ds\quad\forall 0\le y\le\3\notag\\
\le&c_2\delta\3+(1+\eta^2)\int_0^r\left|\frac{1+h_1(s)^2}{\lambda (g_1(s)-sh_1(s))}-\frac{1+h_2(s)^2}{\lambda (g_2(s)-sh_2(s))}\right|\,ds\quad\forall 0\le y\le\3
\end{align}
where
\begin{equation*}
c_2=2\eta\left(\frac{4(n-1)}{3\mu}+\frac{2(1+\eta^2)}{|\lambda|\mu}\right)+\frac{16(n-1)(1+\eta^2)}{9\mu^2}.
\end{equation*}
Now by \eqref{hi-gi-bd},
\begin{align}\label{h-g-bd20}
&\left|\frac{1+h_1(s)^2}{g_1(s)-sh_1(s)}-\frac{1+h_2(s)^2}{g_2(s)-sh_2(s)}\right|\notag\\
\le &4\frac{\left|(1+h_1(s)^2)(g_2(s)-sh_2(s))-(1+h_2(s)^2)(g_1(s)-sh_1(s))\right|}{\mu^2}\quad\forall 0\le s\le\3
\end{align}
and
\begin{align}\label{h-g-bd21}
&\left|(1+h_1(s)^2)(g_2(s)-sh_2(s))-(1+h_2(s)^2)(g_1(s)-sh_1(s))\right|\notag\\
\le&|h_1(s)-h_2(s)||h_1(s)+h_2(s)||g_2(s)-sh_2(s)|+(1+h_2(s)^2)|g_2(s)-g_1(s)+sh_1(s)-sh_2(s)|\notag\\
\le&c_3\delta\quad\forall 0\le s\le\3
\end{align}
\begin{equation*}
c_3=2\left(\eta\left(\frac{5\mu}{4}+\eta\right)+1+\eta^2\right).
\end{equation*}
By \eqref{phi2-difference}, \eqref{h-g-bd20} and \eqref{h-g-bd21},
\begin{equation}\label{phi2-contraction}
|\Phi_2(g_1,h_1)(y)-\Phi_2(g_2,h_2)(y)|\le\left(c_2+\frac{4(1+\eta^2)c_3}{|\lambda|\mu^2}\right)\delta\3\quad\forall 0\le y\le\3.
\end{equation}
We now let 
\begin{equation*}
\ve_2=\min\left(\ve_1,\frac{1}{2}\left(c_2+\frac{4(1+\eta^2)c_3}{|\lambda|\mu^2}\right)^{-1}\right)
\end{equation*}
and $0<\ve\le\ve_2$.
Then by \eqref{phi1-contraction} and \eqref{phi2-contraction},
\begin{equation*}
\|\Phi(g_1,h_1)-\Phi(g_2,h_2)\|_{\cX_\ve}\le\frac{1}{2}\|(g_1,h_1)-(g_2,h_2)\|_{\cX_\ve}\quad\forall (g_1,h_1),(g_2,h_2)\in \cD_{\ve,\eta}.
\end{equation*}
Hence $\Phi$ is a contraction map on $\cD_{\ve,\eta}$. Then by the Banach fix point theorem the map $\Phi$ has a unique fix point. Let $(g,h)\in \cD_{\ve,\eta}$ be the unique fixed point of the map $\Phi$. Then
\begin{equation*}
\Phi(g,h)=(g,h).
\end{equation*}
Hence 
\begin{equation}\label{g-eqn}
g(y)=\mu+\int_0^y h(s)\,ds\quad\Rightarrow\quad g'(y)=h(y)\quad\forall 0<y<\3\quad
\mbox{ and }\quad g(0)=\mu
\end{equation}
and
\begin{align}\label{h-eqn}
&h(y)=\int_0^y (1+h(s)^2)\left(\frac{n-1}{g(s)}-\frac{1+h(s)^2}{\lambda (g(s)-sh(s))}\right)\,ds\quad\forall 0<y<\3\notag\\
\Rightarrow\quad&h'(y)=(1+h(y)^2)\left(\frac{n-1}{g(y)}-\frac{1+h(y)^2}{\lambda (g(y)-sh(y))}\right)\quad\forall 0<y<\3.
\end{align}
By \eqref{h-g-lower-bd}, \eqref{g-eqn} and \eqref{h-eqn}, $g\in C^1([0,\3))\cap C^2(0,\3)$ satisfies 
\eqref{ode-ivp3} and \eqref{f-structure-ineqn3} with $y_0=\3$. Letting $y\to 0$ in \eqref{ode-ivp3} we get
\eqref{r-2nd-derivative-x=0} and hence $g\in C^2([0,\3))$ and the lemma follows.
\end{proof}

By an argument similar to the proof of Lemma \ref{short-time-existence-ode} we have the following lemma.

\begin{lem}\label{existence-ode-extension-lem}
For any $n\ge 2$, $\lambda\ne 0$, $\mu>0$, $M_1>0$, $\delta_0>0$, $r_0,r_1\in\R$, satisfying
\begin{equation*}\label{denominator-lower-bd}
\delta_0\le r_0\le M_1,\quad|r_1|\le M_1,\quad r_0-y_1r_1\ge\delta_0,
\end{equation*}
there exists a constant $\delta_1\in (0,y_0/2)$ depending on $\lambda$, $\delta_0$, $y_0$ and $M_1$ such that  for any $y_0/2<y_1<y_0$ the equation 
\begin{equation}\label{ode-ivp4}
\left\{\begin{aligned}
&\frac{r''}{1+r'{}^2}=\frac{n-1}{r}-\frac{1+r'{}^2}{\lambda(r-yr')}\,\,\,,\quad r(y)>0,\quad\mbox{ in }[y_1,y_1+\delta_1)\\
&r(y_1)=r_0,\quad r'(y_1)=r_1\end{aligned}\right.
\end{equation}
has a unique solution $r(y)\in C^2([y_1,y_1+\delta_1))$ which satisfies
\begin{equation}\label{f-structure-ineqn4}
r(y)>yr'(y)\quad\mbox{ in }[y_1,y_1+\delta_1).
\end{equation}
\end{lem}

\begin{lem}\label{r-monotone-lemma}
Let $n\ge 2$, $0<\lambda\ne\frac{1}{n-1}$, $\mu>0$ and $y_0>0$. Suppose $r(y)\in C^2([0,y_0))$ is the solution of \eqref{ode-ivp3} which satisfies \eqref{f-structure-ineqn3}. Then the following holds.
\begin{enumerate}

\item[(i)] If $\lambda>\frac{1}{n-1}$, then 
\begin{equation*}
r'(y)>0\quad\forall 0<y<y_0.
\end{equation*}

\item[(ii)]  If $0<\lambda<\frac{1}{n-1}$, then 
\begin{equation*}
r'(y)<0\quad\quad\forall 0<y<y_0.
\end{equation*}

\end{enumerate} 
\end{lem}
\begin{proof}
By Lemma \ref{short-time-existence-ode}, \eqref{r-2nd-derivative-x=0} holds. We divide the proof into two cases:

\noindent\textbf{Case 1}: $\lambda>\frac{1}{n-1}$

\noindent By \eqref{r-2nd-derivative-x=0}, $r''(0)>0$. Hence there exists a constant $\delta>0$ such that $r'(s)>0$ for any $0<s<\delta$. Let $(0,a_1)$, $\delta\le a_1\le y_0$, be the maximal interval such that 
\begin{equation*}
r'(s)>0\quad\forall 0<s<a_1.
\end{equation*} 
Suppose $a_1<y_0$. Then $r'(a_1)=0$ and hence $r''(a_1)\le 0$. On the other hand by \eqref{ode-ivp3},
\begin{equation*}
r''(a_1)=\left(n-1-\frac{1}{\lambda}\right)\frac{1}{r(a_1)}>0
\end{equation*} 
and contradiction arises. Hence $a_1=y_0$ and (i) follows.

\noindent\textbf{Case 2}: $0<\lambda<\frac{1}{n-1}$

\noindent By \eqref{r-2nd-derivative-x=0}, $r''(0)<0$. Hence there exists  a constant $\delta>0$ such that $r'(s)<0$ for any $0<s<\delta$. Let $(0,a_1)$, $\delta\le a_1\le y_0$, be the maximal interval such that 
\begin{equation*}
r'(s)<0\quad\forall 0<s<a_1.
\end{equation*} 
Suppose $a_1<y_0$. Then $r'(a_1)=0$ and hence $r''(a_1)\ge 0$. On the other hand by \eqref{ode-ivp3},
\begin{equation*}
r''(a_1)=\left(n-1-\frac{1}{\lambda}\right)\frac{1}{r(a_1)}<0
\end{equation*} 
and contradiction arises. Hence $a_1=y_0$ and (ii) follows.
\end{proof}

\begin{lem}\label{r-monotone-lemma2}
Let $n\ge 2$, $\lambda>\frac{1}{n-1}$, $\mu>0$ and $y_0>0$. Suppose $r(y)\in C^2([0,y_0))$ is the solution of \eqref{ode-ivp3} which satisfies \eqref{f-structure-ineqn3}. Then there exist a constant $\delta_1>0$  such that
\begin{equation}\label{f-structure-ineqn5}
r(y)-yr'(y)\ge\delta_1\quad\forall 0<y<y_0.
\end{equation}
\end{lem}
\begin{proof}
Let $w(y)=r(y)-yr'(y)$, $a_1=\min_{0\le y\le y_0/2}w(y)$, $a_2=\frac{\mu}{\lambda(n-1)}$ and $a_3=\frac{1}{2}\min (a_1,a_2)$. Then $a_1>0$ and $a_3>0$.  By Lemma \ref{r-monotone-lemma}, 
\begin{equation}\label{r-lower-bd}
r(y)\ge\mu\quad\forall 0<y<y_0.
\end{equation}
Suppose there exists 
$y_1\in (y_0/2,y_0)$ such that $w(y_1)<a_3$. Let $(a,b)$ be the maximal interval containing $y_1$ such that $w(y)<a_3$ for any $y\in (a,b)$. Then $a>y_0/2$, $w(a)=a_3$ and 
\begin{equation}\label{w-upper-bd}
w(y)<\frac{\mu}{2\lambda(n-1)}\quad\forall a<y<b.
\end{equation} 
By \eqref{ode-ivp3},  \eqref{r-lower-bd}, \eqref{w-upper-bd}, and a direct computation,
\begin{align*}
w'(y)=&y(1+r'(y)^2)\left(\frac{1+r'(y)^2}{\lambda w(y)}-\frac{n-1}{r(y)}\right)\quad\forall 0<y<y_0\notag\\
\ge&y(1+r'(y)^2)\left(\frac{1}{2\lambda w(y)}+\left(\frac{1}{2\lambda w(y)}-\frac{n-1}{\mu}\right)\right)\quad\forall a<y<b\notag\\
\ge&\frac{y_0}{4\lambda w(y)}>0\quad\forall a<y<b.\notag
\end{align*}
Hence 
\begin{equation*}
w(y)>w(a)=a_3\quad\forall a<y<b
\end{equation*}
and contradiction arises.  Hence no such $y_1$ exists and $w(y)\ge a_3$ for any $y\in (0,y_0)$. Thus \eqref{f-structure-ineqn5} holds with $\delta_1=a_3$.

\end{proof}

\begin{lem}\label{r'-upper-bd-lemma}
Let $n\ge 2$, $\lambda>\frac{1}{n-1}$, $\mu>0$ and $y_0>0$. Suppose $r(y)\in C^2([0,y_0))$ is the solution of \eqref{ode-ivp3} which satisfies \eqref{f-structure-ineqn3}. Then  there exists  a constant $M_1>0$  such that
\begin{equation}\label{r'-upper-bd-a}
0<r'(y)\le M_1\qquad\forall 0<y<y_0
\end{equation}
and
\begin{equation}\label{r-upper-bd2}
\mu\le r(y)\le \mu+M_1y_0\quad\forall 0<y<y_0.
\end{equation}
\end{lem}
\begin{proof}
By \eqref{ode-ivp3}, \eqref{f-structure-ineqn3} and Lemma \ref{r-monotone-lemma},
\begin{equation}\label{r''-r'-ineqn1}
\frac{r''}{1+r'{}^2}\le\frac{n-1}{r}\le\frac{n-1}{\mu}\quad\forall 0<y<y_0.
\end{equation}
Integrating \eqref{r''-r'-ineqn1} over $(0,y_0)$,
\begin{equation}\label{r'-upper-bd11}
\tan^{-1}(r'(y))\le\frac{(n-1)y_0}{\mu}\quad\forall 0<y<y_0.
\end{equation} 
By Lemma \ref{r-monotone-lemma} and \eqref{r'-upper-bd11}, \eqref{r'-upper-bd-a} holds with 
\begin{equation*}
M_1=\tan\left(\frac{(n-1)y_0}{\mu}\right).
\end{equation*}
By \eqref{r'-upper-bd-a} we get \eqref{r-upper-bd2} and the lemma follows.
\end{proof}

\begin{lem}\label{r''-sign-lemma}
Let $n\ge 2$, $\lambda>\frac{1}{n-1}$, $\mu>0$ and $y_0>0$. Suppose $r(y)\in C^2([0,y_0))$ is the solution of \eqref{ode-ivp3} which satisfies \eqref{f-structure-ineqn3}. Then either
\begin{equation}\label{r''-+ve}
r''(y)>0\quad\forall 0<y<y_0
\end{equation}
or there exists a constant $y_1\in (0,y_0)$ such that $r''(y_1)=0$ and
\begin{equation}\label{r''-change-sign}
\left\{\begin{aligned}
&r''(y)>0\quad\forall 0<y<y_1\\
&r''(y)<0\quad\forall y_1<y<y_0
\end{aligned}\right.
\end{equation}
\end{lem}
\begin{proof}
We will use a modification of the proof of Lemma 15 of \cite{DLW} to prove this lemma. By \eqref{r-2nd-derivative-x=0}, $r''(0)>0$. Hence there exists a constant $\delta>0$ such that $r''(s)>0$ for any $0<s<\delta$. Let $(0,y_1)$, $\delta\le y_1\le y_0$, be the maximal interval such that 
\begin{equation*}
r''(s)>0\quad\forall 0<s<y_1.
\end{equation*} 
If $y_1=y_0$, then \eqref{r''-+ve} holds.
If $y_1<y_0$, then $r''(y_1)=0$. By Lemma \ref{r-monotone-lemma} and \eqref{ode-ivp3}, 
\begin{align}
\frac{r'''(y)}{1+r'(y)^2}=&\frac{2r'(y)r''(y)^2}{(1+r'(y)^2)^2}-\frac{n-1}{r(y)}r'(y)-\frac{2r'(y)r''(y)}{\lambda(r(y)-yr'(y))}-\frac{y(1+r'(y)^2)r''(y)}{\lambda(r(y)-yr'(y))^2}\,\forall 0<y<y_0\label{r''-eqn}\\
\Rightarrow\,\frac{r'''(y_1)}{1+r'(y_1)^2}=&-(n-1)\frac{r'(y_1)}{r(y_1)^2}<0\notag
\end{align}  
Hence there exist a constant $0<\delta'<y_0-y_1$ such that $r''(y)<0$ for any $y_1<y<y_1+\delta'$.  Let $(y_1,z_0)$ be the maximal interval such that 
\begin{equation*}
r''(s)<0\quad\forall y_1<s<z_0.
\end{equation*} 
If $z_0<y_0$, then $r''(z_0)=0$ and $r''(z_0)\ge 0$. On the other hand by Lemma \ref{r-monotone-lemma} and \eqref{r''-eqn}, 
\begin{equation*}
\frac{r'''(z_0)}{1+r'(z_0)^2}=-(n-1)\frac{r'(z_0)}{r(z_0)^2}<0
\end{equation*} 
and contradiction arises. Hence $z_0=y_0$ and \eqref{r''-change-sign} follows.
\end{proof}

We are now ready to prove Theorem \ref{existence_soln-thm2}.

\noindent{\bf Proof of Theorem \ref{existence_soln-thm2}}: 
By Lemma \ref{short-time-existence-ode} there exists a constant $y_0'>0$ such that \eqref{ode-ivp3} has a unique solution $r(y)\in C^2([0,y_0'))$ which satisfies \eqref{r-2nd-derivative-x=0} and \eqref{f-structure-ineqn3} in $(0,y_0')$. 
Let $(0,y_0)$ be the maximal interval of existence of solution $r(y)\in C^2([0,y_0))$ of \eqref{ode-ivp3} which satisfies \eqref{f-structure-ineqn3} and \eqref{r-2nd-derivative-x=0}. Suppose $y_0<\infty$.
By Lemma \ref{existence-ode-extension-lem}, Lemma \ref{r-monotone-lemma2} and Lemma \ref{r'-upper-bd-lemma}, there exists a constant $\delta_1\in (0,y_0)$ such that for any $y_0/2<y_1<y_0$ there exists a unique solution $r_1(y)\in C^2([y_1,y_1+\delta_1))$ of \eqref{ode-ivp4} which satisfies \eqref{f-structure-ineqn4} with $r_0=r(y_1)$ and $r_1=r'(y_1)$. Let $y_1\in \left(y_0-\frac{\delta_1}{2},y_0\right)$ and let $r_1(y)\in C^2([y_1,y_1+\delta_1))$ be the unique solution of \eqref{ode-ivp4} given by Lemma \ref{existence-ode-extension-lem} which satisfies \eqref{f-structure-ineqn4}  with $r_0=r(y_1)$ and $r_1=r'(y_1)$. We then extend $r(y)$ to a solution of \eqref{ode-ivp2} in $(0,y_1+\delta_1)$ by setting $r(y)=r_1(y)$ for any $y_0\le y<y_1+\delta_1$. Since $y_1+\delta_1>y_0$, this contradicts the maximality of the interval $(0,y_0)$. Hence $y_0=\infty$ and there exists a unique solution $r(y)\in C^2([0,\infty))$ of the equation \eqref{ode-ivp2} which satisfies \eqref{f-structure-ineqn2} and \eqref{r-2nd-derivative-x=0} and the theorem follows.
{\hfill$\square$\vspace{6pt}}

\noindent{\bf Proof of Theorem \ref{asymptotic-behaviour-thm}}: We will give a simple proof different from the sketchy proof of this result in \cite{DLW} here. By (i) of Lemma \ref{r-monotone-lemma} \eqref{r-monotone-increase} holds. By Lemma \ref{r''-sign-lemma} either 
\begin{equation}\label{r''--ve}
r''(y)>0\quad\forall y>0
\end{equation}
or there exists $y_1>0$ such that \eqref{r''-sign} holds. Suppose \eqref{r''--ve} holds. Then
\begin{equation}\label{a1-defn}
a_1:=\lim_{y\to\infty}r'(y)\quad\mbox{ exists }
\end{equation}
and $a_1>0$. We now divide the proof into two cases.

\noindent{\bf Case 1}: $a_1=\infty$

\noindent Then there exists $y_2>0$ such that
\begin{equation}\label{r'-bd-below15}
r'(y)>\sqrt{2(n-1)\lambda}\quad\forall y>y_2
\end{equation}
By \eqref{ode-ivp2} and \eqref{r'-bd-below15},
\begin{align*}
\frac{r''(y)}{1+r'(y)^2}\le&\frac{1}{r(y)}\left(n-1-\frac{1+r'(y)^2}{\lambda}\right)\qquad\qquad\,\,\forall y>0\notag\\
\le&\frac{1}{r(y)}\left(n-1-\frac{1+2(n-1)\lambda}{\lambda}\right)<0\quad\forall y>y_2
\end{align*}
which contradicts \eqref{r''--ve}. Hence $a_1\ne\infty$.

\noindent{\bf Case 2}: $a_1<\infty$

\noindent By \eqref{f-structure-ineqn2},
\begin{equation}\label{y-r'/r-ratio-bd}
0<\frac{yr'(y)}{r(y)}<1\quad\forall y>0.
\end{equation}
Now by \eqref{a1-defn} and the l'Hosiptal rule,
\begin{equation}\label{y-r'/r-ratio}
\lim_{y\to\infty}\frac{r(y)}{y}=\lim_{y\to\infty}r'(y)=a_1\quad
\Rightarrow\quad\lim_{y\to\infty}\frac{yr'(y)}{r(y)}=\frac{\lim_{y\to\infty}r'(y)}{\lim_{y\to\infty}r(y)/y}=1.
\end{equation}
By \eqref{ode-ivp2}, \eqref{a1-defn}, \eqref{y-r'/r-ratio-bd}, \eqref{y-r'/r-ratio} and the l'Hosiptal rule,
\begin{align*}
\lim_{y\to\infty}\frac{r(y)r''(y)}{1+a_1^2}
=&\lim_{y\to\infty}\frac{r(y)r''(y)}{1+r'(y)^2}\notag\\
=&n-1-\frac{1+a_1^2}{\lambda}\cdot\frac{1}{\lim_{y\to\infty}\left(1-\frac{yr'(y)}{r(y)}\right)}\notag\\
=&-\infty
\end{align*}
which contradicts \eqref{r''--ve}. Hence $a_1<\infty$ does not hold. Thus by Case 1 and Case 2 \eqref{r''--ve} cannot hold. Hence there exists $y_1>0$ such that \eqref{r''-sign} holds.
 
By \eqref{r''-sign} and Lemma \ref{r-monotone-lemma}, \eqref{r'-infty-+ve} holds.  By \eqref{r-monotone-increase},
\begin{equation*}
a_2:=\lim_{y\to\infty}r(y)\in (\mu,\infty]\quad\mbox{ exists}.
\end{equation*}
Since by \eqref{r''-sign} $(r(y)-yr'(y))'=-yr''(y)>0$ for any $y>y_1$,
\begin{equation}\label{r-yr'-limit}
a_3:=\lim_{y\to\infty}(r(y)-yr'(y))\in (r(y_1)-y_1r'(y_1),\infty]\quad\mbox{ exists}.
\end{equation}
Suppose 
\begin{equation}\label{a2-finite}
a_2\in (\mu,\infty).
\end{equation}
Then
\begin{equation}\label{a1=0}
a_1=0.
\end{equation}
By \eqref{r'-infty-+ve}, \eqref{r-yr'-limit} and \eqref{a2-finite},
\begin{equation*}
a_4:=\lim_{y\to\infty}yr'(y)=a_2-a_3\in [0,a_2-r(y_1)+y_1r'(y_1))\quad\mbox{ exists}.
\end{equation*}
Suppose $a_4>0$. Then there exists $y_2>y_1$ such that
\begin{align*}
yr'(y)\ge &a_4/2\quad\forall y\ge y_2\\
\Rightarrow\qquad\,\, r(y)\ge& r(y_2)+\frac{a_4}{2}\log (y/y_2)\quad\forall y\ge y_2\\
\Rightarrow\qquad\quad a_2= &\infty
\end{align*}
which contradicts \eqref{a2-finite}. Hence 
\begin{equation}\label{a4=0}
a_4=0.
\end{equation}
Letting $y\to\infty$ in \eqref{ode-ivp2}, by \eqref{a1=0} and \eqref{a4=0},
\begin{equation*}
\lim_{y\to\infty}r''(y)=\left(n-1-\frac{1}{\lambda}\right)\frac{1}{a_2}>0
\end{equation*}
which contradicts \eqref{r''-sign}. Hence  \eqref{a2-finite} does not hold and $a_2=\infty$. Thus \eqref{r-infty-behaviour} holds and the theorem follows.

{\hfill$\square$\vspace{6pt}}

\end{document}